\documentclass[10pt]{amsart}
\usepackage{amssymb, amsmath, amsthm, amsfonts, amscd}
\usepackage{xcolor}
\usepackage{datetime}
\usepackage{mathabx}
\usepackage{enumerate}
\newtheorem{theorem}{Theorem}[section]
\newtheorem{lemma}[theorem]{Lemma}
\newtheorem{proposition}[theorem]{Proposition}
\newtheorem{corollary}[theorem]{Corollary}
\theoremstyle{definition}
\newtheorem{definition}[theorem]{Definition}
\newtheorem{example}[theorem]{Example}

\newtheorem{question}[theorem]{Question}

\theoremstyle{remark}
\newtheorem{remark}[theorem]{Remark}
\numberwithin{equation}{section}

\title[Spectral decomposition of Normal $\mathcal{AM}$-operators]{Spectral decomposition of Normal Absolutely minimum attaining operators}
\author{Neeru Bala}
\address{Department of Mathematics, Indian Institute of Technology - Hyderabad, Kandi, Sangareddy, Telangana, India 502 285.}
\email{ma16resch11001@iith.ac.in}

\author{G. Ramesh}
\address{Department of Mathematics, Indian Institute of Technology - Hyderabad, Kandi, Sangareddy, Telangana, India 502 285.}
\email{rameshg@iith.ac.in}

\date{\currenttime ;  \today}
\begin{document}

	\begin{abstract}

Let $T:H_1\rightarrow H_2$ be a bounded linear operator defined between complex Hilbert spaces $H_1$ and $H_2$. We say  $T$ to be \textit{minimum attaining} if there exists a unit vector $x\in H_1$ such that $\|Tx\|=m(T)$, where $m(T):=\inf{\{\|Tx\|:x\in H_1,\; \|x\|=1}\}$ is the \textit{minimum modulus} of $T$.
We say $T$ to be \textit{absolutely minimum attaining} ($\mathcal{AM}$-operators in short), if for any closed subspace $M$ of $H_1$ the restriction operator $T|_M:M\rightarrow H_2$ is minimum attaining.
		In this paper, we give a new characterization of positive absolutely minimum attaining operators ($\mathcal{AM}$-operators, in short), in terms of its essential spectrum. Using this we obtain a sufficient condition under which the adjoint of an $\mathcal{AM}$-operator is $\mathcal{AM}$. 
		We show that a paranormal absolutely minimum attaining operator is  hyponormal. Finally, we establish a spectral decomposition of normal absolutely minimum attaining operators. In proving all these results we prove several spectral results for paranormal operators. We illustrate our main result with an example.
	\end{abstract}

	\maketitle
	\section{Introduction}

 The class of minimum attaining operators on Hilbert spaces is first introduced by Carvajal and Neves \cite{CAR1}. An important property of this class of operators is that it is dense in  $\mathcal B(H_1,H_2)$, the space of all bounded linear operators between $H_1$ and $H_2$ with respect to the operator norm \cite{SHKGR}. For densely defined closed operators in Hilbert spaces, this class of operators is studied in \cite{SCH}.

An important subclass of the minimum attaining operators is the class of absolutely minimum attaining operators, which was introduced by Carvajal and Neves in \cite{CAR1}. We say  $T$ is an absolutely minimum attaining operator or $\mathcal{AM}$-operator,  if for every closed subspace  $M$ of $H$ the restriction $T\big|_M:M\rightarrow H_2$ is minimum attaining.

The minimum attaining operators and the absolutely minimum attaining operators are defined analogous to those of norm attaining and the absolutely norm attaining operators, respectively. Recall that $T\in \mathcal B(H_1,H_2)$ is \textit{norm attaining operator} if there exists a $x\in H_1$ with $\|x\|=1$ such that $\|Tx\|=\|T\|$ and \textit{absolutely norm attaining} or $\mathcal{AN}$-operator, if for every closed subspace $N$ of $H_1$ the operator $T\big|_{N}:N\rightarrow H_2$ is norm attaining.

There is a close connection between the minimum attaining operators and norm attaining operators. Similarly, absolutely minimum attaining operators and the absolutely norm attaining operators are related with each other.
	
A partial characterization of positive, absolutely norm attaining operators was given in \cite{CAR1}, which was further improved in \cite{PAL,RAM}. Another characterization of positive, absolutely norm attaining operators is discussed in \cite{RAM1}.

In the present article we give a characterization for positive absolutely minimum attaining operators. Similar to absolutely norm attaining operators, adjoint of an absolutely minimum attaining operator need not be absolutely minimum attaining. Here we will give a condition similar to the one given in \cite{RAM1}, for the adjoint of an absolutely minimum attaining operator to be absolutely minimum attaining.
	
	Next we will study the class of paranormal, $\mathcal{AM}$-operators. In particular, we will show that the class $\{T\in\mathcal{AM}(H): T\text{ is paranormal with }N(T)=N(T^*)\}$ is same as the class of hyponormal $\mathcal{AM}$-operators. We prove a spectral decomposition theorem for normal $\mathcal{AM}$-operators. Our results can be compared to those of paranormal $\mathcal{AN}$-operators given in \cite{RAM1}. Specifically, we will show that if $T$ is normal $\mathcal{AM}$-operator then there exists pairs $(H_{\beta}, U_{\beta})$, where $H_\beta$ is a reducing subspace for $T$, $U_\beta\in \mathcal{B}(H_\beta)$ is a unitary such that,
	\begin{enumerate}
		\item $H=\underset{\beta\in\sigma(|T|)}{\oplus}H_{\beta}$,
		\item $T=\underset{\beta\in\sigma(|T|)}{\oplus}\beta U_{\beta}$.
	\end{enumerate}
	
	There are four sections in this article. In the second section we will provide the basic definitions and results which we will be using through out the article.
	
	In the third section we will give a characterization for positive $\mathcal{AM}$-operators and a sufficient condition for the adjoint of an $\mathcal{AM}$-operator to be $\mathcal{AM}$. Fourth section consists of spectral decomposition of normal $\mathcal{AM}$-operators.
	\section{Preliminaries}
	Throughout this article we consider complex Hilbert spaces which are denoted by $H,H_1,H_2$ etc. Mostly we assume that these spaces are infinite dimensional.  We denote the space of all bounded linear operators from $H_1$ to $H_2$ by $\mathcal{B}(H_1,H_2)$ and  in case if $H_1=H_2=H$, we denote this by $\mathcal{B}(H)$. By an operator we mean a  linear operator all  the time.  If $T\in \mathcal{B}(H_1,H_2)$, the adjoint of $T$ is denoted by $T^*\in \mathcal{B}(H_2,H_1)$. An operator $T\in \mathcal{B}(H)$ is called \textit{normal} if $TT^*=T^*T$, \textit{unitary} if $TT^*=I=T^*T$, \textit{self-adjoint} if $T=T^*$ and \textit{positive} if $T$ is self-adjoint and $\langle Tx,x\rangle\geq 0$ for all $x\in H$. If $S$ and $T$ are two self adjoint operators in $\mathcal{B}(H)$ then $S\leq T$ if and only if $\langle Sx,x\rangle\leq\langle Tx,x\rangle$ for all $x\in H$. Let $\mathcal A\subseteq \mathcal B(H)$, then we denote the set of all self-adjoint elements (operators) of $\mathcal A$ by $\mathcal A_{sa}$ and the set of all positive elements (operators) of $\mathcal A$ by $\mathcal A_{+}$. The set of all positive bounded linear operators on $H$ is denoted by $\mathcal{B}(H)_+$.

	Let $H_1,H_2$ be two Hilbert spaces. Then $H_1\oplus H_2=\{(h_1,h_2):h_1\in H_1,h_2\in H_2\}$ is a Hilbert space with the inner product $\langle.,.\rangle$, given by
	\begin{equation*}
	\langle (h_1,h_2),(k_1,k_2)\rangle=\langle h_1,k_1\rangle_{H_1}+\langle h_2,k_2\rangle_{H_2}
	\end{equation*}
	for $h_1,k_1\in H_1$ and $h_2,k_2\in H_2$.

Let $T_i\in \mathcal B(H_i),\; i=1,2$. Define $T_1\oplus T_2:H_1\oplus H_2\rightarrow H_1\oplus H_2$ by
\begin{equation*}
(T_1\oplus T_2)(x_1,x_2)=(T_1x_1,T_2x_2),\; \text{for all}\; x_i\in H_i,\; i=1,2.
\end{equation*}
This can be represented by a $2\times 2$ operator matrix by $T_1\oplus T_2=\left(
                                                                    \begin{array}{cc}
                                                                      T_1 & 0 \\
                                                                      0 & T_2 \\
                                                                    \end{array}
                                                                  \right).$
                                                                  	
	For $T\in \mathcal{B}(H_1,H_2)$, $N(T)$ and $R(T)$ denote the null space and range space of $T$, respectively. If $M$ is a closed subspace of a Hilbert space $H$, then $M^{\perp}$ is the orthogonal complement of $M$, $T\big|_{M}$ denote the restriction of the operator $T$ to $M$ and the orthogonal projection onto $M$ in $H$, is denoted by $P_M$. The unit sphere of $M$ is $S_M=\{x\in M:\|x\|=1\}$.
	
	An operator $T\in \mathcal{B}(H_1,H_2)$ is said to be  \textit{finite-rank}, if $R(T)$ is finite dimensional. The space of all finite-rank operators in $\mathcal{B}(H_1,H_2)$ is denoted by $\mathcal{F}(H_1,H_2)$ and $\mathcal{F}(H,H)=\mathcal{F}(H)$. If $T\in \mathcal{B}(H_1,H_2)$, then $T$ is said to be compact, if for every bounded set $A\subseteq H_1$, the set $T(A)\subseteq H_2$ is pre-compact. The space of all compact operators in $\mathcal{B}(H_1,H_2)$, is denoted by $\mathcal{K}(H_1,H_2)$.

If $V\in \mathcal B(H)$, then $V$ is called an \textit{isometry} if $\|Vx\|=\|x\|$ for all $x\in H$. We say $V$ to be a \textit{co-isometry} if $V^*$ is an isometry. Equivalently, $V$ is an isometry if and only if $V^*V=I$ and co-isometry if and only if $VV^*=I$.
	
	For $T\in \mathcal{B}(H)$, the set $\sigma(T)=\{\lambda\in\mathbb{C}:T-\lambda I\,\text{ is not invertible in }\mathcal{B}(H)\}$ is called the \textit{spectrum} of $T$. Note that $\sigma(T)$ is a non-empty compact subset of $\mathbb{C}$. The spectrum of $T\in \mathcal B(H)$ decomposes as the disjoint union of the \textit{point spectrum}, $\sigma_p(T)$, the \textit{continuous spectrum}, $\sigma_c(T)$ and the \textit{residual spectrum} $\sigma_r(T)$, where
	\begin{align*}
	\sigma_p(T)=&\{\lambda\in\mathbb{C}:T-\lambda I \text{ is not injective}\},\\
	\sigma_r(T)=&\{\lambda\in\mathbb{C}:T-\lambda I\text{ is injective but }R(T-\lambda I)\text{ is not dense in }H\},\\
	\sigma_c(T)=&\sigma(T)\setminus\left(\sigma_p(T)\cup\sigma_r(T)\right).
	\end{align*}
	If $T$ is normal, then $\sigma_r(T)$ is empty. For a self-adjoint operator $T\in \mathcal{B}(H)$, the spectrum can be divided into disjoint union of the discrete spectrum and the essential spectrum. We summarize these details here.
	\begin{theorem}\cite[Theorem 7.10,7.11, Page 236]{REED}\label{thm8}
	 Let $T=T^*\in \mathcal{B}(H)$, the spectrum $\sigma(T)$ of $T$ decomposes as the disjoint union of the essential spectrum, $\sigma_{ess}(T)$ and the discrete spectrum of $T$, $\sigma_{d}(T)$, where $\sigma_{d}(T)$, is the set of isolated finite multiplicity eigenvalues of $T$. Furthermore, $\lambda\in\sigma_{ess}(T)$ if and only if one of the following holds:
	 \begin{enumerate}
	 \item $\lambda$ is an eigenvalue of infinite multiplicity.
	 \item $\lambda$ is a limit point of $\sigma_p(T)$.
	 \item $\lambda\in\sigma_c(T)$.
	 \end{enumerate}
 \end{theorem}

More details about essential spectrum and discrete spectrum can be found in \cite[Page 235, 236]{REED}.
	Now we will quote two results from \cite{TAL}, namely [Theorem 5.2, Page 288] and [Theore 5.4, Page 289]. These results are proved for unbounded operators in \cite{TAL}, but they are true for bounded operators as well. A bounded version of these results is given below.
	
	\begin{theorem}\label{thm7}
		Let $T_i\in\mathcal{B}(H_i),\,i=1,2$ and $T=T_1\oplus T_2$. Then
		\begin{enumerate}
			\item $N(T)=N(T_1)\oplus N(T_2).$
			\item $R(T)=R(T_1)\oplus R(T_2).$
			\item $T^{-1}$ exists if and only if $T_1^{-1}$ and $T_2^{-1}$ exists.
			\item $R(T)=H_1\oplus H_2$ if and only if $R(T_1)=H_1$ and $R(T_2)=H_2$.
		\end{enumerate}
	\end{theorem}
	\begin{theorem}\label{thm9}
		Let $T$ be as defined in Theorem (\ref{thm7}). Then
		\begin{enumerate}
			\item $\sigma(T)=\sigma(T_1)\cup\sigma(T_2)$.
			\item $\sigma_p(T)=\sigma_p(T_1)\cup\sigma_p(T_2)$.
			If it is further assumed that $\sigma(T_1)$ and $\sigma(T_2)$ have no point in common, it follows that
			\item $\sigma_c(T)=\sigma_c(T_1)\cup\sigma_c(T_2)$.
			\item $\sigma_r(T)=\sigma_r(T_1)\cup\sigma_r(T_2)$.
		\end{enumerate}
	\end{theorem}
  \begin{definition}\cite{BER}
 Let $T\in\mathcal{B}(H_1,H_2)$. Then
 \begin{enumerate}
 	\item The \textit{minimum modulus} of $T$ is $m(T):=\inf\{\|Tx\|:x\in S_{H_1}\}$.
 	\item The \textit{essential minimum modulus} of $T$ is $m_e(T):=\inf\{|\lambda|:\lambda\in\sigma_{ess}(T)\}$.
 \end{enumerate}
 \end{definition}

 It is to be noted that $m(T)>0$ if and only if $R(T)$ is closed and $T$ is one-to-one. In particular, if $H_1=H_2=H$ and $T$ is normal, then $m(T)>0$ if and only if $T^{-1}$ exists and $T^{-1}\in\mathcal{B}(H)$.

\begin{definition}\cite[Definition 1.1]{CAR1}
Let $T\in \mathcal{B}(H_1,H_2)$. Then
\begin{enumerate}
	\item  $T$ is norm attaining, if there exist $x\in S_{H_1}$ such that $\|T\|=\|Tx\|$, where $\|T\|=\sup\{\|Tx\|:x\in H_1,\|x\|=1\}$.
	\item  $T$ is absolutely norm attaining operator, if for every closed subspace $M\subseteq H_1$, $T\big|_M:M\rightarrow H_2$ is norm attaining operator.
\end{enumerate}
\end{definition}

 We denote the classes of norm attaining and absolutely norm attaining operators in $\mathcal{B}(H_1,H_2)$ by $\mathcal{N}(H_1,H_2)$ and $\mathcal{AN}(H_1,H_2)$, respectively. In particular, $\mathcal N(H,H)$ and $\mathcal{AN}(H,H)$ are denoted by $\mathcal N(H)$ and $\mathcal{AN}(H)$, respectively.

\begin{definition}\cite{CAR1}
	Let $T\in \mathcal{B}(H_1,H_2)$. Then
	\begin{enumerate}
		\item $T$ is minimum attaining, if there exist $x\in S_{H_1}$ such that $m(T)=\|Tx\|$.
		\item $T$ is absolutely minimum attaining operator, if for every closed subspace $M\subseteq H_1$, $T\big|_M$ is minimum attaining operator.
	\end{enumerate}
\end{definition}

  We denote the classes of minimum attaining and absolutely minimum attaining operators in $\mathcal{B}(H_1,H_2)$ by $\mathcal{M}(H_1,H_2)$ and $\mathcal{AM}(H_1,H_2)$, respectively. In particular, $\mathcal M(H,H)$ and $\mathcal{AM}(H,H)$ are denoted by $\mathcal M(H)$ and $\mathcal{AM}(H)$, respectively.

  \begin{theorem}\cite[Theorem 5.9]{GAN1}\label{thm6}
  	Let $H$ be a complex Hilbert space of arbitrary dimension and let $P$ be a positive operator on $H$. Then $P$ is an $\mathcal{AM}$-operator if and only if $P$ is of the form $P = \beta I -K +F$, where $\beta \geq 0,K\in\mathcal{K}(H)_+$ with $\|K\|\leq\beta$ and $F\in\mathcal{F}(H)_+$, satisfying $KF=FK=0$.
  \end{theorem}

  Let $T\in \mathcal B(H_1,H_2)$. For $y\in H_2$, we say that $u\in H_1$ is a least square solution of the equation $Tx=y$, if $\|Tu-y\|\leq\|Tx-y\|$ for all $x\in H_1$. Moreover, if $u_0$ is a least square solution of $Tx=y$ and $\|u_0\|\leq \|u\|$ for any least square solution $u$ of $Tx=y$, then $u_0$ is called \textit{the least square solution of minimal norm}. Such a solution is unique.
  \begin{definition}\cite[ Page 223]{GRO1}
 Suppose $T\in\mathcal{B}(H_1,H_2)$ and let $D(T^{\dagger})=R(T)\oplus R(T)^{\perp}$, where $D(T^{\dagger})$ is domain of $T^{\dagger}$. The \textit{generalized inverse} or the \textit{Moore-Penrose inverse} of $T$ is the operator $T^{\dagger}:D(T^{\dagger})\to H_1$ which assigns to each $b\in D(T^{\dagger})$ the unique least square solution of minimal norm of the equation $Tx=b$.
\end{definition}
  If $T$ has closed range, then $T^{\dagger}$ is the unique operator in $\mathcal{B}(H_2,H_1)$ satisfying $TT^{\dagger}=P_{R(T)}$ and $T^{\dagger}T=P_{R(T^{\dagger})}.$

Here we list out some of the properties of the Moore-Penrose inverse, which we need to prove our results. The following theorem is true for densely defined closed operators in a Hilbert space, which is also true for bounded operators. We state it for bounded operators.
\begin{theorem}\cite[Theorem 2, Page 341]{BEN}\label{thmben}
	Let $T \in\mathcal{B}(H_1,H_2)$. Then
	\begin{enumerate}
		\item $ R(T^{\dagger}) = N(T)^{\perp}$.
		\item $ N(T^{\dagger}) = R(T)^{\perp} = N(T^*)$
		\item  $T^\dagger$ is continuous if and only if $R(T)$ is closed.
		\item $(T^{\dagger})^{\dagger} = T$.
		\item $(T^*)^{\dagger}=(T^{\dagger})^*$.
		\item $N((T^*)^{\dagger})=N(T)$.
		\item $T^{\dagger}(T^*)^{\dagger}$ is positive and $(T^*T)^{\dagger}=T^{\dagger}(T^*)^{\dagger}$.
		\item $(T^*)^{\dagger}T^{\dagger}$ is positive and $(TT^*)^{\dagger}=(T^*)^{\dagger}T^{\dagger}$.
	\end{enumerate}
\end{theorem}

	\section{$\mathcal{AM}$-operators }
In this section we give a new characterization of positive $\mathcal{AM}$-operators in terms of the essential spectrum. Using this we give a sufficient condition under which the adjoint of an $\mathcal{AM}$-operator is again an $\mathcal{AM}$-operator.

	First, we have the following observation.
	\begin{remark}\label{REM1}
		Let $T\in \mathcal{AM}(H)_+$. By Theorem (\ref{thm6}), $T=\beta I-K+F$, where $\beta\geq0$, $K\in \mathcal{K}(H)_+$, $F\in \mathcal{F}(H)_+$, $KF=0=FK$ and $\|K\|\leq \beta$. Then we have the following:
		\begin{enumerate}
			\item $\beta=0$ if and only if $T$ is finite-rank.
			\item $\sigma_{ess}(T)=\{\beta\}$.
			\item $\beta=0$ if and only if $N(T)$ is infinite dimensional.
					\end{enumerate}
	\end{remark}
	Here we will give a relation between the essential spectrum of a self-adjoint invertible operator and its inverse.
\begin{lemma}\label{lem1}
Let $T\in \mathcal{B}(H)$ be self-adjoint and $T^{-1}\in \mathcal{B}(H)$. Then $\lambda\in\sigma_{ess}(T)$ if and only if $\lambda^{-1}\in\sigma_{ess}(T^{-1})$.
\end{lemma}
\begin{proof}
	It is enough to prove that $\lambda\in\sigma_{ess}(T)$ implies $\lambda^{-1}\in\sigma_{ess}(T^{-1})$, as $(T^{-1})^{-1}=T$.
	
We know that $\sigma(T^{-1})=\{\lambda^{-1}:\lambda\in\sigma(T)\}$ and $\sigma_p(T^{-1})=\{\lambda^{-1}:\lambda\in\sigma_p(T)\}$.
If $\lambda\in\sigma_{ess}(T)$ then either $\lambda$ is an eigenvalue of $T$ with infinite multiplicity or $\lambda$ is a limit point of $\sigma_p(T)$ or $\lambda\in\sigma_c(T)$.

 If $\lambda$ is an eigenvalue of $T$ with infinite multiplicity, then $1/\lambda$ is an eigenvalue of $T^{-1}$ of infinite multiplicity with same eigenvectors. Hence $\lambda^{-1}\in \sigma_{ess}(T^{-1})$.

 If $\lambda$ is a limit point of $\sigma_p(T)$, then there exist a sequence $(\lambda_n)\subseteq\sigma_p(T)$ such that $\lambda_n$ converges to $\lambda$. Since $0\notin\sigma(T)$, so $\lambda_n,\lambda\neq 0$ for every $n\in \mathbb{N}$ and hence $1/\lambda_n$ converges to $1/\lambda.$ So $1/\lambda\in\sigma_{ess}(T^{-1})$ as $1/\lambda_n\in\sigma_p(T^{-1})$.

 If $\lambda\in\sigma_c(T)$ i.e. $\lambda\notin \sigma_p(T)$, so $1/\lambda\notin\sigma_p(T^{-1})$ and hence $1/\lambda\in\sigma_c(T^{-1})$, because $\sigma_r(T^{-1})$ is empty. Hence $1/\lambda\in\sigma_{ess}(T^{-1}).\qedhere$
\end{proof}
\begin{remark}\label{REM4}
	Let $T\in\mathcal{B}(H)$. If $N(T)$ be a reducing subspace for $T$, then $\sigma(T_0)\subseteq\sigma(T)\subseteq\sigma(T_0)\cup\{0\}$, where $T_0=T\big|_{N(T)^{\perp}}$.
\end{remark}
\begin{proof}
	This is clear from (1) of Theorem (\ref{thm9}).
\end{proof}
\begin{lemma}\label{lemma2}
	Let $T\in\mathcal{B}(H)$ be self adjoint and $T_0=T\big|_{{N(T)^{\perp}}}$. Then
	\begin{enumerate}
		\item $\sigma_{ess}(T_0)\subseteq\sigma_{ess}(T)\subseteq\sigma_{ess}(T_0)\cup\{0\}$.
		\item $\sigma_{d}(T_0)\subseteq\sigma_{d}(T)\subseteq\sigma_{d}(T_0)\cup\{0\}$.
	\end{enumerate}
\end{lemma}
\begin{proof} Since $N(T)$ is reducing, we have
	\[
	T=
	\begin{bmatrix}
	0 & 0\\
	0 & T_0
	\end{bmatrix}
	.\]
	\begin{enumerate}
		\item Note that as $T$ is self-adjoint, $N(T)$ reduces $T$. If $\lambda\in\mathbb{R}$, then
			\[
		T-\lambda I=
		\begin{bmatrix}
		-\lambda I_{N(T)} & 0\\
		0 & T_0-\lambda I_{N(T)^{\perp}}
		\end{bmatrix}.
		\]
		 Using Theorem (\ref{thm7}), we have
		\begin{equation}\label{eq1}
		N(T-\lambda I)=N\left(-\lambda I_{N(T)}\right)\oplus N\left(T_0-\lambda I_{N(T)^{\perp}}\right).
		\end{equation}
		and
		$ R(T-\lambda I)=N(T)\cup R(T_0-\lambda I_{N(T)^{\perp}})$. Thus we can conclude that $\sigma_p(T_0)\subseteq\sigma_p(T)$ and $\sigma_c(T_0)\subseteq\sigma_c(T)$.
	
		Let $\lambda\in\sigma_{ess}(T_0)$. By Theorem (\ref{thm8}), one of the following hold
		\begin{enumerate}
			\item $\lambda\in\sigma_c(T_0)$.
			\item $\lambda$ is an eigenvalue of $T_0$ of infinite multiplicity.
			\item $\lambda$ is a limit point of $\sigma_p(T_0)$.
		\end{enumerate}
	Using above argument, we get $\lambda\in\sigma_{ess}(T)$. Second containment follows from Theorem (\ref{thm9}), which says that $\sigma_p(T)=\sigma_p(T_0)\cup\{0\}$ and $\sigma_c(T)=\sigma_c(T_0)\cup\{0\}$.
	\item From Equation (\ref{eq1}), we get $\sigma_d(T)\subseteq\sigma_d(T_0)\cup\{0\}$. Other containment follows from Remark (\ref{REM4}) and Equation (\ref{eq1}).$\qedhere$
	\end{enumerate}
\end{proof}
Next we will prove that Lemma (\ref{lem1}) can be generalized to $T^{\dagger}$.
\begin{lemma}\label{lem2}
Let $T\in \mathcal{B}(H)$. Suppose $T$ is a self-adjoint operator with closed range. Let $0\neq\lambda\in\mathbb{R}$. Then
\begin{enumerate}
\item $\lambda\in\sigma_{ess}(T)$ if and only if $1/\lambda\in\sigma_{ess}(T^{\dagger})$.
\item $0\in\sigma_{ess}(T)$ if and only if $0\in\sigma_{ess}(T^{\dagger})$.
\end{enumerate}
\end{lemma}
\begin{proof}
Let $T_0=T\big|_{N(T)\bot}$. It is enough to show one way implication in both the cases, as $(T^{\dagger})^{\dagger}=T$.
\begin{enumerate}
	\item As $\lambda\neq 0$, by Lemma (\ref{lemma2}), $\lambda\in\sigma_{ess}(T)$ if and only if $\lambda\in\sigma_{ess}(T_0)$. Using Lemma (\ref{lem1}), we get  $1/\lambda\in\sigma_{ess}(T_0^{-1})$. Hence $1/\lambda\in\sigma_{ess}(T^{\dagger})$, by $(2)$ of Lemma (\ref{lemma2}).
	\item Let $0\in\sigma_{ess}(T)$. As $T$ is self-adjoint and  $R(T)$ is closed so $0\notin\sigma_{c}(T)$ and by \cite[Theorem 4.4]{KUL}, $0$ is not a limit point of $\sigma_p(T)$. So $0$ is an eigenvalue of $T$ with infinite multiplicity. As $T$ is self-adjoint, we have $N(T)=N(T^{\dagger})$. So $0$ is an eigenvalue of $T^{\dagger}$ with infinite multiplicity. Hence $0\in\sigma_{ess}(T^{\dagger}).\qedhere$
\end{enumerate}
\end{proof}
\begin{remark}\label{REM3}
Let $T$ be as in Lemma (\ref{lem2}). Define
\begin{align*}
\lambda^{\dagger}=
\begin{cases}
\lambda^{-1}&\text{if }\lambda\neq0,\\
0&\text{if }\lambda=0.\\
\end{cases}
\end{align*}
\begin{enumerate}
	\item By Lemma (\ref{lem2}), we can conclude that $\lambda\in\sigma_{ess}(T)$ if and only if $\lambda^{\dagger}\in\sigma_{ess}(T^{\dagger})$.
	\item Since $\sigma(T)=\sigma_d(T)\cup\sigma_{ess}(T)$, we can also conclude that $\lambda\in\sigma_d(T)$ if and only if $\lambda^{\dagger}\in\sigma_d(T^{\dagger})$.
\end{enumerate}
\end{remark}
\begin{proposition}\label{REM2}
	Let $T\in\mathcal{AM}(H_1)_+$ and $H_2$ be a finite dimensional Hilbert space. If $S\in B(H_2)_+$, then $S\oplus T\in\mathcal{AM}(H_2\oplus H_1)_+$.
\end{proposition}
\begin{proof}
	Let $T\in\mathcal{AM}(H_1)_+$. By Theorem (\ref{thm6}), $T=\beta I_{H_1}-K+F$, where $\beta\geq 0$, $K\in\mathcal{K}(H_1)_+$ with $\|K\|\leq\beta$ and $F\in\mathcal{F}(H_1)_+$ satisfying $KF=FK=0$. Then
	\begin{align*}
    S\oplus T & =
	\begin{bmatrix}
	S & 0\\
	0 & \beta I_{H_1}-K+F
	\end{bmatrix}\\
	& =
	{\beta I-
	\begin{bmatrix}
	0 & 0\\
	0 & K
	\end{bmatrix}
	+
	\begin{bmatrix}
	S-\beta I_{H_2} & 0\\
	0 & F
	\end{bmatrix}}\\
	& =\beta I- \tilde{K}+\tilde{F},
	\end{align*}
	where
	\[
	\tilde{K}=
	\begin{bmatrix}
	0 & 0\\
	0 & K
	\end{bmatrix}
	\text{ and }
	\tilde{F}=
	\begin{bmatrix}
	S-\beta I_{H_2} & 0\\
	0 & F
	\end{bmatrix}.
	\]
	Again by using Theorem (\ref{thm6}), we get $S\oplus T\in\mathcal{AM}(H_2\oplus H_1)_+.\qedhere$
	\end{proof}
By similar arguments as above, we can prove the following Remark for $\mathcal{AN}$-operators.
\begin{remark}
	Let $T\in\mathcal{AN}(H_1)_+$ and $H_2$ be a finite dimensional Hilbert space. If $S\in B(H_2)_+$, then $S\oplus T\in\mathcal{AN}(H_2\oplus H_1)_+$.
\end{remark}
Now we will generalize the result of \cite[Theorem 5.1]{KUL1}, to any bounded linear operator, by dropping the injectivity condition.
\begin{theorem}\label{thm2}
Let $T\in \mathcal{B}(H)$ be a positive operator. Then $T\in \mathcal{AM}(H)$ if and only if $R(T)$ is closed and $T^{\dagger}\in \mathcal{AN}(H)$.
\end{theorem}
\begin{proof}
Let $T\in \mathcal{AM}(H)$ and $T_0=T\big|_{N(T)^\perp}$. Then by \cite[Proposition 3.3]{GAN}, $R(T)$ is closed and by Theorem (\ref{thm6}), $T=\beta I-K+F$ where $\beta\geq0$, $K\in \mathcal{K}(H)_+$, $F\in \mathcal{F}(H)_+$ satisfying $KF=0=FK$ and $\|K\|\leq \beta$. We consider the following two cases which exhaust all possibilities.

Case (1): Let $N(T)$ be infinite dimensional. We get $\beta=0$ and $T$ is finite-rank operator. By \cite[Theorem 3.2]{KAR}, $T^{\dagger}$ is finite-rank and hence $T^{\dagger}\in \mathcal{AN}(H)$.

Case (2): Let $N(T)$ be finite dimensional. Since $T$ is $\mathcal{AM}$-operator, so $T_0\in \mathcal{AM}(N(T)^{\perp})$. By \cite[Theorem 5.1]{KUL1}, $T_0^{-1}\in \mathcal{AN}(N(T)^{\perp})$. Using Proposition \ref{REM2}, we get
\[
T^{\dagger}=
\begin{bmatrix}
0 & 0\\
0 & T_0^{-1}
\end{bmatrix}\in \mathcal{AN}(H).
\]

Conversely, assume that $R(T)$ is closed and $T^{\dagger}\in \mathcal{AN}(H)$. By \cite[Theorem 5.1]{PAL}, we get $T^{\dagger}=\alpha I+K+F$, where $\alpha\geq0,\,K\in \mathcal{K}(H)_+$ and $F\in \mathcal{F}(H)$ is a self-adjoint operator.

Case (1): Let $N(T)$ be infinite dimensional. Then $0$ is an eigenvalue of $T$ with infinite multiplicity. So, $\alpha=0$ and $T^{\dagger}$ is compact. Since restriction of a compact operator to a closed subspace is compact, $T_0^{-1}$ is compact. We know that a compact operator is invertible if its domain is finite dimensional, so $R(T)$ is finite dimensional and hence T is finite-rank operator. Thus $T\in \mathcal{AM}(H)$.

Case (2): Let $N(T)$ be finite dimensional. Since $T^{\dagger}\in \mathcal{AN}(H),\,T_0^{-1}\in \mathcal{AN}(N(T)^{\perp})$. By \cite[Theorem 5.1]{KUL1}, $T_0\in \mathcal{AM}(N(T)^{\perp})$. Since $N(T)$ is finite dimensional, by Proposition \ref{REM2}, $T\in \mathcal{AM}(H).\qedhere$
\end{proof}
\begin{lemma}\label{lemma3}
Let $T\in \mathcal{B}(H)$ be a positive operator. If $\sigma_{ess}(T)$ is singleton and $(m_e(T),\|T\|]$ contains only finitely many eigenvalues of $T$, then $R(T)$ is closed.
\end{lemma}
\begin{proof}
To prove $R(T)$ is closed, it is enough to show that $0$ is not a limit point of $\sigma(T)$, by \cite[Theorem 4.4]{KUL}. On the contrary assume that $0$ is a limit point of $\sigma(T)$. Since $\sigma_{ess}(T)$ is singleton, $0$ must be a limit point of $\sigma_d(T)$. In fact $0$ is a limit point of a decreasing sequence in $\sigma_p(T)$, this implies $\sigma_{ess}(T)=\{0\}$. Hence the sequence which is converging to $0$ is the zero sequence, which is a contradiction. So $0$ is not a limit point of $\sigma(T)$ and hence $R(T)$ is closed.$\qedhere$
\end{proof}
Here we will give a new characterization for positive $\mathcal{AM}$-operators, which is similar to \cite[Theorem 2.4]{RAM1}.
\begin{theorem}\label{thm3}
Let $T\in \mathcal{B}(H)$ be a positive operator. Then $T\in \mathcal{AM}(H)$ if and only if $\sigma_{ess}(T)$ is singleton and $(m_e(T),\|T\|]$ contains only finitely many eigenvalues of $T$.
\end{theorem}
\begin{proof}
Let $\sigma_{ess}(T)=\{\beta\}$ and $\lambda_1,\lambda_2,\ldots\lambda_m$ be the eigenvalues of $T$ contained in $(m_e(T),\|T\|]$. By Lemma (\ref{lemma3}), $R(T)$ is closed.

Case (1): Let $\beta=0$. By Lemma (\ref{lem2}), $\sigma_{ess}(T^{\dagger})=\{0\}$ and $[m(T^{\dagger}),m_e(T^{\dagger})]=\{0\}$. So, by \cite[Theorem 2.4]{RAM1}, $T^{\dagger}\in \mathcal{AN}(H)$. Hence by Theorem (\ref{thm2}), $T\in \mathcal{AM}(H)$.

Case (2): Let $\beta>0$. This implies that $\sigma_{ess}(T^{\dagger})=\{1/\beta\}$ and $[m(T^{\dagger}),1/\beta)$ contains only finitely many eigenvalues of $T^{\dagger}$, namely either $0,1/\lambda_1,1/\lambda_2,\ldots1/\lambda_m$ or $1/\lambda_1,1/\lambda_2,\ldots1/\lambda_m$. By \cite[Theorem 2.4]{RAM1}, $T^{\dagger}\in \mathcal{AN}(H)$. Hence by Theorem (\ref{thm2}), $T\in \mathcal{AM}(H)$.

Conversely, let $T\in \mathcal{AM}(H)$ and $T_0=T\big|_{N(T)^\perp}$. By \cite[Theorem 2.4]{RAM1}, it is enough to show that $T^{\dagger}\in \mathcal{AN}(H)$. Since $T\in \mathcal{AM}(H)$, $T_0\in\mathcal{ AM}(N(T)^\perp)$ and by \cite[Theorem 5.1]{KUL1}, $T_0^{-1}$ is $\mathcal{AN}$-operator. Firstly if $N(T)$ is finite dimensional, then $N(T^{\dagger})$ is finite dimensional and hence $T^{\dagger}$ is an $\mathcal{AN}$-operator, by Proposition (\ref{REM2}). Secondly, if $N(T)$ is infinite dimensional, then $T$ is finite-rank operator. By \cite[Theorem 3.2]{KAR}, $T^{\dagger}$ is finite-rank and hence $\mathcal{AN}$-operator.$\qedhere$
\end{proof}
In general, if $T\in\mathcal{AM}(H)$, then $T^*$ need not be an $\mathcal{AM}$-operator (see \cite{CAR,GAN} for more details). The same is true for $\mathcal{AN}$-operators. A sufficient condition to hold this result for $\mathcal{AN}$-operators is given in \cite{RAM1}, which depends on the essential spectrum. A similar condition works for $\mathcal{AM}$-operators too. The details are given below.
\begin{theorem}\label{thm4}
Let $T\in \mathcal{B}(H)$ and $\sigma_{ess}(T^*T)=\sigma_{ess}(TT^*)$. Then
\begin{enumerate}
	\item $TT^*\in \mathcal{AM}(H)$ if and only if $T^*T\in \mathcal{AM}(H)$.
	\item $T\in \mathcal{AM}(H)$ if and only if $T^*\in \mathcal{AM}(H)$.
\end{enumerate}
\end{theorem}
	\begin{proof}
	By Lemma (\ref{lem2}), $\sigma_{ess}((T^{\dagger})^*T^{\dagger})=\sigma_{ess}(T^{\dagger}(T^{\dagger})^*)$.
	\begin{enumerate}
		\item By Theorem (\ref{thm2}), $TT^*\in \mathcal{AM}(H)$ if and only if $R(TT^*)$ is closed and $(TT^*)^{\dagger}\in \mathcal{AN}(H)$, i.e. $(T^{\dagger})^*T^{\dagger}\in \mathcal{AN}(H)$. Since $\sigma_{ess}((T^{\dagger})^*T^{\dagger})=\sigma_{ess}(T^{\dagger}(T^{\dagger})^*)$, by \cite[Theorem 2.7]{RAM1}, we get $T^{\dagger}(T^{\dagger})^*\in \mathcal{AN}(H)$. Again by applying Theorem (\ref{thm2}), we get $T^*T\in \mathcal{AM}(H)$.
		\item This follows by Case (1) and by \cite[Corollary 4.9]{GAN}, that $T\in \mathcal{AM}(H)$ if and only if $T^*T\in \mathcal{AM}(H).\qedhere$
	\end{enumerate}
	\end{proof}

It is well known that for $T\in B(H),\,\sigma(T^*T)\setminus\{0\}=\sigma(TT^*)\setminus\{0\}$. We can ask whether the same is true if the spectrum is replaced by the essential spectrum. We answer this question affirmatively.
\begin{lemma}\label{lemma8}
	Let $T\in\mathcal{B}(H)$. Then $\sigma_{ess}(T^*T)\setminus\{0\}=\sigma_{ess}(TT^*)\setminus\{0\}$.
\end{lemma}
\begin{proof}
	To prove the result it is enough to show that $\sigma_{ess}(T^*T)\setminus\{0\}\subseteq\sigma_{ess}(TT^*)\setminus\{0\}$. Let $\alpha\in\sigma_{ess}(T^*T)\setminus\{0\}$. By Theorem (\ref{thm8}), either $\alpha$ is an eigenvalue of $T^*T$ with infinite multiplicity or $\alpha\in\sigma_c(T^*T)$ or $\alpha$ is a limit point of $\sigma_p(T^*T)$.
	
	Case (1): Let $\alpha$ be an eigenvalue of $T^*T$ with infinite multiplicity. This implies there exist $\{x_{\delta}\in H:\delta\in\Lambda\}$ such that $T^*Tx_{\delta}=\alpha x_{\delta}$, where $\Lambda$ is an indexing set. So we get $TT^*(Tx_{\delta})=\alpha(Tx_{\delta})$. We have $Tx_{\delta}\neq Tx_{\bar{\delta}}$ if $\delta\neq\bar{\delta}$ and $\delta,\bar{\delta}\in\Lambda$. Because if $Tx_{\delta}=Tx_{\bar{\delta}}$ then $T^*Tx_{\delta}=T^*Tx_{\bar{\delta}}$, thus $x_{\delta}=x_{\bar{\delta}}$. Hence $\alpha$ is an eigenvalue of $TT^*$ of infinite multiplicity.
	
	Case (2): Let $\alpha\in\sigma_c(T^*T)$. This implies $T^*T-\alpha I$ is one-one with $\overline{R(T^*T-\alpha I)}=H$ but $T^*T-\alpha I$ does not have a bounded inverse. We know that $\sigma_r(TT^*)$ is empty and $\sigma(TT^*)\setminus\{0\}=\sigma(T^*T)\setminus\{0\}$. Thus we get that either $\alpha\in\sigma_p(TT^*)$ or $\alpha\in\sigma_c(TT^*)$. Now we will show that $\alpha\notin\sigma_p(TT^*)$. Let $(TT^*-\alpha I)x=0$ for some $0\neq x\in H$. Applying on both sides $T^*$, we get $T^*TT^*x=\alpha T^*x$, thus $T^*x\in N(T^*T-\alpha I)$. As $T^*T-\alpha I$ is one-one, so $T^*x=0$ and hence $x=0$. Thus we get $\alpha\in\sigma_{c}(TT^*)$.
	
	Case (3): Let $\alpha$ be a limit point of $\sigma_p(T^*T)$. So there exist a sequence $(\lambda_n)\subseteq\sigma_p(T^*T)\setminus\{0\}$ such that $(\lambda_n)$ converges to $\alpha$. Also $(\lambda_n)\subseteq\sigma_p(TT^*)$, so $\alpha$ is a limit point of $\sigma_p(TT^*)$. Hence $\alpha\in\sigma_{ess}(TT^*)$.
	
	In all the three cases we get $\alpha\in\sigma_{ess}(TT^*)$. Hence $\sigma_{ess}(T^*T)\setminus\{0\}\subseteq\sigma_{ess}(TT^*)\setminus\{0\}.\qedhere$
\end{proof}
\begin{theorem}\label{spectralequalities}
	Suppose $T\in\mathcal{B}(H)$, such that $N(T)=N(T^*)$. Then we have the following:
\begin{enumerate}
\item\label{ptspectrumequality} $\sigma_p(T^*T)=\sigma_p(TT^*)$.
\item \label{contspectrumequality}$\sigma_c(T^*T)=\sigma_c(TT^*)$.
\item\label{spectrumequality} $\sigma(T^*T)=\sigma(TT^*)$.
\item\label{essspectrumequality} $\sigma_{ess}(T^*T)=\sigma_{ess}(TT^*)$.
\item\label{discretespectrumequality} $\sigma_{d}(T^*T)=\sigma_{d}(TT^*)$.
\end{enumerate}
\begin{proof}
First note that as $N(T^*T)=N(T)$ and $N(TT^*)=N(T^*)$, by the assumption it follows that $N(T^*T)=N(TT^*)$.

Proof of (\ref{ptspectrumequality}): Let $\lambda \in \sigma_p(T^*T)$. First, assume that $\lambda=0$. Then ${\{0}\}\neq N(T^*T)=N(TT^*)$, we can conclude that $0\in \sigma_p(TT^*)$. The other implication follows in the similar lines. Next, assume that $\lambda\neq 0$. Let $0\neq x\in H$ be such that $T^*Tx=\lambda x$. Then $(TT^*)Tx=\lambda Tx$. Since $x\in R(T^*T)\subseteq N(T)^{\bot}$, $Tx\neq 0$. This means that $\lambda \in \sigma_{p}(TT^*)$. Similarly, the other way implication can be proved.

Proof of (\ref{contspectrumequality}): Let $\lambda \in \sigma_c(T^*T)$. Then $\lambda \notin \sigma_p(T^*T)$ and $R(T^*T-\lambda I)$ is not closed. By (\ref{ptspectrumequality}), $\lambda \notin \sigma_p(TT^*)$. By \cite[Theorem 4.4]{KUL}, it follows that $\lambda$ is  an accumulation point of $\sigma(T^*T)$. As $\sigma(T^*T)\setminus {\{0}\}=\sigma(TT^*)\setminus {\{0}\}$, $\lambda$ is an accumulation point of $\sigma(TT^*)$. Hence by \cite[Theorem 4.4]{KUL}, $R(TT^*-\lambda I)$ is not closed, concluding $\lambda \in \sigma_c(TT^*)$. The other implication can be proved with similar arguments.

Proof of (\ref{spectrumequality}): Since for a self-adjoint operator the residual spectrum is empty and the spectrum is disjoint union of the point spectrum, continuous spectrum and the residual spectrum, by (\ref{ptspectrumequality}) and (\ref{contspectrumequality}), the conclusion follows.

Proof of (\ref{essspectrumequality}): In view of Lemma (\ref{lemma8}), it is enough to show that $0\in \sigma_{ess}(T^*T)$ if and only if $0\in \sigma_{ess}(TT^*)$. But this follows by the definition of the essential spectrum and (\ref{ptspectrumequality}) and (\ref{contspectrumequality}) proved above.

Proof of (\ref{discretespectrumequality}): For a self-adjoint operator $A\in \mathcal B(H)$, $\sigma(A)=\sigma_{ess}(A)\cup \sigma_{d}(A)$, the conclusion follows by (\ref{spectrumequality}) and (\ref{essspectrumequality}) above.
\end{proof}
\end{theorem}
\begin{corollary}\label{minmodessminmodequality}
Let $T\in \mathcal B(H)$ be such that $N(T)=N(T^*)$. Then
  \begin{enumerate}
\item \label{equalityofminmod}$m(T)=m(T^*)$.
\item \label{equalitypfessminmod}$m_e(T)=m_e(T^*)$.
\end{enumerate}
\begin{proof}
  Since $\sigma(T^*T)=\sigma(TT^*)$ by (\ref{spectrumequality}) of Theorem (\ref{spectralequalities}), by the spectral mapping theorem, $\sigma(|T|)=\sigma(|T^*|)$. Now,
  \begin{align*}
  m(T)&=\inf{\{\lambda: \lambda \in \sigma(|T|)}\}\\
      &=\inf{\{\lambda: \lambda \in \sigma(|T^*|)}\}\\
      &=m(T^*).
  \end{align*}

  Next, by \cite{BOL} and by (\ref{essspectrumequality}) of Theorem (\ref{spectralequalities}), we have
  \begin{align*}
  m_e(T)&=\inf{\{\lambda:\lambda \in \sigma_{ess}(|T|)}\}\\
        &=\inf{\{\lambda:\lambda \in \sigma_{ess}(|T^*|)}\}\\
        &=m_e(T^*).
  \end{align*}
\end{proof}
\end{corollary}
\begin{corollary}
Let $T\in \mathcal B(H)$ be such that $N(T)=N(T^*)$. Then
 $T\in \mathcal{M}(H)$ if and only if $T^*\in \mathcal {M}(H)$.
\end{corollary}
\begin{proof}
	First, note that $m(T)=m(T^*)$ by (\ref{equalityofminmod}) of Corollary (\ref{minmodessminmodequality}). It is enough to show that $T\in\mathcal{M}(H)$ implies $T^*\in\mathcal{M}(H)$. Assume that $T\in\mathcal{M}(H)$. Thus $m(T)^2\in \sigma_p(T^*T)$. But by Theorem (\ref{spectralequalities}), $m(T^*)^2\in \sigma_p(TT^*)$, concluding $TT^*\in \mathcal M(H)$ and hence $T^*\in \mathcal M(H)$.
\end{proof}
\section{Normal $\mathcal{AM}$-operators}
In this section we describe a spectral decomposition of normal $\mathcal{AM}$-operators.
\begin{definition}
Let $T\in\mathcal{B}(H)$.Then
\begin{enumerate}
	\item $T$ is \textit{hyponormal} if $TT^*\leq T^*T$. Equivalently, $T$ is hyponormal if $\|T^*x\|\leq\|Tx\|$ for all $x\in H$.
	\item $T$ is \textit{paranormal} if $\|Tx\|^2\leq \|T^2x\|$ for all $x\in S_H$. Equivalently, $T$ is paranormal if $\|Tx\|^2\leq \|T^2x\|\|x\|$ for all $x\in H$.
\end{enumerate}
\end{definition}

 It is easy to see that every hyponormal operator is paranormal \cite{IST1}. More details about hyponormal and paranormal operators can be found in \cite{AND,BER,IST,IST1,HAL}.
 \begin{remark}\label{paraAM}
 	Let $T\in\mathcal{B}(H)$.
 	\begin{enumerate}
 		\item If $T$ is paranormal, then $N(T)=N(T^2)$.
 		\item If $T$ is hyponormal, then $N(T)\subset N(T^*)$. This inclusion is strict, for example, the right shift operator $R$ on $l^2(\mathbb{N})$ is hyponormal such that $N(R)\subsetneq N(R^*)$. Also note that $R^*$ is not paranormal.
 	\end{enumerate}
 \end{remark}
\begin{theorem}\label{lem5}
Let $T\in \mathcal{B}(H)$, with $N(T)=N(T^*)$. Then
\begin{enumerate}
	\item\label{para} $T$ is paranormal if and only if $T^{\dagger}$ is paranormal.
	\item\label{AM} $T\in\mathcal{AM}(H)$ if and only if $R(T)$ is closed and $T^{\dagger}\in\mathcal{AN}(H)$.
\end{enumerate}
 \end{theorem}
\begin{proof}
	Let $T_0=T|_{N(T)^{\perp}}:N(T)^{\perp}\to N(T)^{\perp}$.
	
	Proof of (\ref{para}): Let $T$ be paranormal. By \cite[Lemma 3.8]{RAM1}, $T_0$ is paranormal. Since $N(T)=N(T^*)$, we have $T_0^{-1}$ is paranormal, by \cite[Theorem 1]{IST}. As
		\begin{align*}
		T^{\dagger}x=
		\begin{cases}
		 0&\text{ if }x\in N(T),\\
		 T_0^{-1}x &\text{ if }x\in N(T)^{\perp}.
		\end{cases}
		\end{align*}
		Hence $T^{\dagger}$ is paranormal.
		Reverse implication is clear, as $(T^{\dagger})^\dagger=T$.
		
	Proof of (\ref{AM}): By Theorem (\ref{spectralequalities}) $\sigma_{ess}(TT^*)=\sigma_{ess}(T^*T)$. Thus using Theorem (\ref{thm4}) $T^*T\in\mathcal{AM}(H)$ if and only if $TT^*\in\mathcal{AM}(H)$. Now using all these arguments and Theorem (\ref{thm2}) we conclude the following.
	\begin{align*}
	T\in\mathcal{AM}(H)&\iff T^*T\in\mathcal{AM}(H)\\
	&\iff TT^*\in\mathcal{AM}(H)\\
	&\iff R(TT^*) \text{ is closed and }(TT^*)^{\dagger}\in\mathcal{AN}(H)\\
	&\iff R(T) \text{ is closed and }(T^{\dagger})^*T^{\dagger}\in\mathcal{AN}(H)\\
	&\iff R(T) \text{ is closed and }T^{\dagger}\in\mathcal{AN}(H).
	\end{align*}
		
		
		
	

\end{proof}

The following result is not used in the article, but it is of independent interest.
\begin{theorem}
	Let $T\in \mathcal{B}(H)$ be such that $R(T)=R(T^2)$. Also assume that $R(T)$ is closed. Then $T$ is paranormal implies $T^{\dagger}$ is paranormal.
\end{theorem}
\begin{proof}
	Assume that $T$ is paranormal. Then
	\begin{equation*}
	\|Tx\|^2\leq\|T^2x\|\|x\|,\, \forall\,x\in H.
	\end{equation*}
	 Since $T^{\dagger}u=0$ for all $u\in R(T)^{\perp}$. It suffices to show that
	 \begin{equation*}
	 \|T^{\dagger}y\|^2\leq\|{T^{\dagger}}^2y\|\|y\|,\,\forall\, y\in R(T).
	 \end{equation*}
	  Let $y\in R(T)=R(T^2)$. Then $y=T^2v$ for some $v\in N(T^2)^{\perp}=N(T)^{\perp}$ (see (\ref{para}) of Remark (\ref{paraAM})). Now $T^{\dagger}y=T^{\dagger}T^2v=P_{N(T)^{\perp}}Tv=Tv$, as $R(T)\subseteq N(T)^{\perp}$. Also $(T^{\dagger})^2y=T^{\dagger}Tv=P_{R(T^{\dagger})}v=v$. Thus we have
	\begin{equation*}
	\|T^{\dagger}y\|^2=\|Tv\|^2\leq \|v\|\|y\|=\|{T^{\dagger}}^2y\|\|y\|.
	\end{equation*}
	Thus $T^{\dagger}$ is paranormal.$\qedhere$
\end{proof}

 Now we will show that the classes of paranormal $\mathcal{AM}$-operators and hyponormal $\mathcal{AM}$-operators are the same, under some assumption.
\begin{theorem}\label{hyponormalorpara}
	Let $T\in\mathcal{B}(H)$ with $N(T)=N(T^*)$. If $T\in\mathcal{AM}(H)$ is paranormal then  $T$ is hyponormal.
\end{theorem}
\begin{proof}
	Since $T\in\mathcal{AM}(H)$, $T^*T\in \mathcal{AM}(H)$. Hence by Theorem (\ref{thm6}), $T^*T=\beta I-K+F$, where $K\in\mathcal{K}(H)_+$ with $\|K\|\leq \beta$ and $F\in\mathcal{F}(H)_+$ satisfies $KF=0=FK$.
	
	Case (1): Let $\beta=0$. As a consequence of Remark (\ref{REM1}), we conclude that $T^*T$ is a finite-rank operator and so is $T$. Also we know that a paranormal compact operator is normal \cite{QIU}, hence $T$ is normal.

	Case (2): Let $\beta> 0$. By Theorem (\ref{spectralequalities}), we have $\sigma_{ess}(T^*T)=\{\beta\}=\sigma_{ess}(TT^*)$. Now using Theorem (\ref{thm4}), we get $T^*\in\mathcal{AM}(H)$. So $H$ has a basis, say $\{y_{\alpha}:\alpha\in\Lambda\}$,  consisting of eigenvectors of $TT^*$. Let $TT^*y_{\alpha}=\lambda_{\alpha}y_{\alpha}$ where $\lambda_{\alpha}\in [0,\infty)$.

If $\lambda_{\alpha_0}\neq 0$ for some $\alpha_0$, then
\begin{align*}
\|T^*y_{\alpha_0}\|^4&=|\langle T^*y_{\alpha_0},T^*y_{\alpha_0}\rangle|^2\\
&=|\langle TT^*y_{\alpha_0},y_{\alpha_0}\rangle|^2\\
&\leq\|TT^*y_{\alpha_0}\|^2\|y_{\alpha_0}\|^2.
\end{align*}
As $T$ is paranormal, by the above inequality, we have
\begin{align*}
\|T^*y_{\alpha_0}\|^4\leq \|TT^*y_{\alpha_0}\|^2\|y_{\alpha_0}\|^2&\leq \|T^2T^*y_{\alpha_0}\|\|T^*y_{\alpha_0}\|\|y_{\alpha_0}\|^2\\
&=\lambda_{\alpha_0}\|Ty_{\alpha_0}\|\|T^*y_{\alpha_0}\|\|y_{\alpha_0}\|^2\\
&=\|Ty_{\alpha_0}\|\|T^*y_{\alpha_0}\|\langle TT^*y_{\alpha_0},y_{\alpha_0}\rangle\\
&=\|Ty_{\alpha_0}\|\|T^*y_{\alpha_0}\|\|T^*y_{\alpha_0}\|^2.
\end{align*}
This implies $\|T^*y_{\alpha_0}\|\leq\|Ty_{\alpha_0}\|$.

If for some $\alpha,\,\lambda_{\alpha}=0$, then $\|T^*y_{\alpha}\|^2=0=\|Ty_{\alpha}\|^2$, since $N(T)=N(T^*)$. This implies for each $y\in H$, $\|T^*y\|^2\leq \|Ty\|^2$. Hence $T$ is hyponormal.


\end{proof}

By a similar argument as in Theorem (\ref{hyponormalorpara}), we can prove the following;

\begin{theorem}
Let $T\in \mathcal{AN}(H)$ be paranormal operator with $N(T)=N(T^*)$. Then $T$ is hyponormal.
\end{theorem}

 Now we will give a characterization of normal $\mathcal{AM}$-operators, which is similar to the result for paranormal $\mathcal{AN}$-operators, given in \cite{RAM1}.
\begin{theorem}\label{spectraldecomposition}
	Let $T\in\mathcal{B}(H)$ be normal. Suppose $T\in\mathcal{AM}(H)$ with $\Lambda=\sigma(|T|)$, where $|T|=\sqrt{T^*T}$. Then there exist $(H_{\beta}, U_{\beta})_{\beta\in\Lambda}$, where $H_\beta$ is a reducing subspace for $T$, $U_\beta\in \mathcal{B}(H_\beta)$ is a unitary such that,
	\begin{enumerate}
		\item $H=\underset{\beta\in\Lambda}{\oplus}H_{\beta}$,
		\item $T=\underset{\beta\in\Lambda}{\oplus}\beta U_{\beta}$.
	\end{enumerate}
\end{theorem}
\begin{proof}
	As $T\in\mathcal{AM}(H)$ is normal, $T^{\dagger}$ is normal $\mathcal{AN}$-operator, by Theorem (\ref{thmben}). Using \cite[Theorem 3.9]{RAM1}, there exist $(G_{\alpha}, V_{\alpha})_{\alpha\in\Gamma}$, where $G_{\alpha}$ is a reducing subspace for $T^{\dagger}$, $V_{\alpha}\in \mathcal{B}(G_{\alpha})$ is an unitary such that,
	\begin{enumerate}[(i)]
		\item\label{H} $H=\underset{\alpha\in\Gamma}{\oplus}G_{\alpha}$,
		\item\label{T} $T^{\dagger}=\underset{\alpha\in\Gamma}{\oplus}\alpha V_{\alpha}$,
	\end{enumerate}
	where $\Gamma=\sigma(|T^{\dagger}|)$. Now we claim that $G_{\alpha}$ is a reducing subspace for $T$ as well. If $\alpha=0$ then $G_{\alpha}=N(T^*)=N(T)$. Clearly $G_{\alpha}$ is a reducing subspace for $T$. On the other hand if $\alpha\neq 0$, then $G_{\alpha}\subseteq N(T)^{\perp}$.
	
	Let $x\in G_{\alpha}$. Assume that $Tx=a+b$, where $a\in G_{\alpha}\cap N(T)^{\perp}$ and $b\in G_{\alpha}^{\perp}\cap N(T)^{\perp}$. Now
	\begin{align*}
	T^{\dagger}b=T^{\dagger}Tx-Ta &=P_{R(T^{\dagger})}x-Ta\\
	&=P_{N(T)^{\perp}}x-Ta\\
	&=x-Ta\in G_{\alpha}^{\perp}.
	\end{align*}
	Since $G_{\alpha}$ is a reducing subspace for $T^{\dagger}$, we get $T^{\dagger}b\in G_{\alpha}^{\perp}$, thus $T^{\dagger}b\in G_{\alpha}\cap G_{\alpha}^{\perp}$. So $T^{\dagger}b=0$, but $b\in N(T)^{\perp}$, so $b=0$ and hence $Tx=a\in G_{\alpha}$. We get $G_{\alpha}$ is invariant under $T$.
	
	To show that $G_{\alpha}^{\perp}$ is invariant under $T$, let $y\in G_{\alpha}^{\perp}$. If $Ty=u+v$, where $u\in G_{\alpha}$ and $v\in G_{\alpha}^{\perp}$, then $T^{\dagger}Ty=P_{N(T)^{\perp}}y\in N(T)^{\perp}\cap G_{\alpha}^{\perp}$. Now by similar argument as above, we get $T^{\dagger}u\in G_{\alpha}\cap G_{\alpha}^{\perp}$, so $T^{\dagger}u=0$. This implies $u\in N(T)$, but $u\in G_{\alpha}\subseteq N(T)^{\perp}$, so $u=0$ and hence $Ty=v\in G_{\alpha}^{\perp}$. By (\ref{H}) and (\ref{T}), we get
	\begin{enumerate}
		\item $T=\underset{\alpha\in\Gamma}{\oplus}\alpha^{\dagger} V_{\alpha}^*$.
	\end{enumerate}
	where $\alpha^{\dagger}$ is defined as in Remark (\ref{REM3}). By using \cite[Proposition 3.15]{KUL1}, we get $\alpha^{\dagger}\in\sigma(|T^{\dagger}|^{\dagger})=\sigma(|T^*|)$. But 
	by Theorem (\ref{spectralequalities}), we have that $\sigma(|T|)=\sigma(|T^*|)$, as $N(T)=N(T^*)$.
	Thus we get the result, by taking $\beta=\alpha^{\dagger}$, $H_{\beta}=G_{\alpha}$ and $U_{\beta}=V_{\beta}^*$. As $V_{\beta}$ is an unitary, it is clear that $U_{\beta}$ is a unitary.
\end{proof}

Below we illustrate Theorem (\ref{spectraldecomposition}) with an example.
\begin{example}\label{multplication_AMoperator}
		Let $(X,\Sigma,\mu)$ be a $\sigma$-finite measure space. For $f\in L_{\infty}(X)$ define the multiplication operator $M_f:L_2(X)\to L_2(X) $ by
	\begin{equation*}
	M_f(g)=fg,\text{ for all }g\in L_2(X).
	\end{equation*}
	Then
	\begin{enumerate}
		\item $m(M_f)=\text{ess inf}(f)$, where $\text{ess inf}(f)=\sup\{\alpha\in\mathbb{R}:\mu(\{x\in X:|f(x)|<\alpha\})=0\}$.
		\item $M_f$ is minimum attaining if and only if there exist $A\in \Sigma$  with $\mu(A)>0$ such that $|f(t)|=\text{ess inf} (f)$ for all $t\in A$.
		\item $M_f\in\mathcal{AM}(L_2(X))$ if and only if there exist a sequence $(A_i)\subseteq \Sigma$ with $\mu(A_i)>0$ such that  $|f(t)|=\text{ess inf}(f_i)$ for all $t\in A_i$ and $X=X_0\cup\left(\underset{i=0}{\overset{\infty}{\cup}}A_i\right)$, where $X_0\in\Sigma$ with $\mu(X_0)=0$ and $f_i=f$ with domain $X\setminus\left(\cup_{j=1}^{i-1}A_j\right)$.
	\end{enumerate}
\end{example}
\begin{proof}
	\begin{enumerate}
		\item  Let $g\in L_2(X)$. Then
		\begin{align*}
		\|M_f(g)\|_2^2=&\int_{X}|f(t)|^2|g(t)|^2d\mu(t)\\
		\geq&(\text{ess inf}(f))^2\|g\|^2_2.
		\end{align*}
		Thus we get $m(M_f)\geq \text{ess inf}(f)$. Now for every $n\in\mathbb{N}$, define $$E_n=\{t\in X:|f(t)|\leq \text{ess inf}(f)+1/n\}.$$ By the definition of $\text{ess inf}(f)$, it is clear that $\mu(E_n)>0$. As $(X,\Sigma,\mu)$ is a $\sigma$-finite measure space, choose a measurable set $F_n\subseteq E_n$ such that $0<\mu(F_n)<\infty$, for every $n\in\mathbb{N}$. Let $g_n=\frac{\chi_{F_n}}{\sqrt{\mu(F_n)}}$, then
		\begin{align*}
		\|M_f(g_n)\|_2^2=&\int_{X}|f(t)|^2\frac{\chi_{F_n}}{\mu(F_n)}d\mu\\
		=&\frac{1}{\mu(F_n)}\int_{F_n}|f(t)|^2d\mu(t)\\
		\leq&(\text{ess inf}(f)+1/n)^2,\,\forall\,n\in\mathbb{N.}
		\end{align*}
		Thus we get $m(M_f)\leq \text{ess inf}(f)$. Hence $\|M_f\|=\text{ess inf}(f)$.
		\item  To prove this, we use a similar technique that is used in \cite[Lemma 2.6]{acostaetal}. First we assume that $M_f$ is a minimum attaining operator. Then there exists a $ g_0\in S_{L_2(X)}$, such that $\|M_f(g_0)\|_2=m(M_f)=\text{ess inf}(f)$. Now we have
		\begin{align*}
		\int_{X}(\text{ess inf}(f))^2|g_0(t)|^2d\mu(t)=&(\text{ess inf}(f))^2\\
		=&\|M_f(g_0)\|_2^2\\
		=&\int_{X}|f(t)|^2|g_0(t)|^2d\mu(t).
		\end{align*}
		This implies $\text{ess inf}(f)|g_0(t)|=|f(t)||g_0(t)|$ a.e. Let $C=\{t\in X:g_0(t)\neq 0\}$. Then $C$ is a measurable set with $\mu(C)>0$. There exist a subset $Z$ of $C$ such that $\mu(Z)=0$. Thus $|f(t)|=\text{ess inf}(f)$ for all $t\in A:=C\setminus Z$.
		
		Conversely, let $A\in \Sigma$ satisfying the given property. Choose a measurable subset $B$ of $A$ such that $0<\mu(B)<\infty$. Define $g_0=\frac{\chi_B}{\sqrt{\mu(B)}}$. It is easy to see that $g_0\in S_{L_2(X)}$. Now
		\begin{align*}
		\|M_f(g_0)\|_2^2=&\int_{X}|fg_0|^2d\mu\\
		=&\int_{B}|fg_0|^2d\mu\\
		=&\frac{(\text{ess inf} (f))^2}{\mu(B)}\int_{B}d\mu\\
		=&(\text{ess inf}(f))^2.
		\end{align*}
		Hence $M_f$ is a minimum attaining operator.
		\item First assume that $M_f\in\mathcal{AM}(L_2(X))$. As $M_f$ is minimum attaining, we get $ A_0\in\Sigma$ such that $\mu(A_0)>0$ and $|f(t)|=\text{ess inf}(f),$ for all $ t\in A_0$.
		
		Let $X_1=X\setminus A_0$. If $\mu(X_1)=0$ then $M_{|f|}=M_{\text{ess inf}(f)}$. Otherwise, it is easy to see that $L_2(X_1)=\{\widehat{g}|_{X_1}: \widehat{g}\in L_2(X),\,\widehat{g}=0\text{ a.e. on }\,A_0 \}$. Define $\Sigma_1=\{X_1\cap A:A\in\Sigma\}$ and $\mu_1=\mu|_{\Sigma_1}$. Let $f_1=f|_{X_1}\in L_{\infty}(X_1)$ and define $M_{f_1}:L_2(X_1)\to L_2(X_1)$ as $M_{f_1}g=f_1g$, for $g\in L_2(X_1)$.
		
		Let $G=\{g\in L_2(X):g=0\text{ a.e. on }\,A_0\}$. We claim that $G$ is a closed subspace of $L_2(X)$. To prove our claim, let $(\widehat{g}_n)\subseteq G$ and $\widehat{g}_n$ converges to some $h\in L_2(X)$. As $(\widehat{g}_n)$ is a Cauchy sequence in $G$,  $(g_n)$ is a Cauchy sequence in $L_2(X_1)$, where $g_n=\hat{g_n}|_{X_1}$ for every $n\in\mathbb{N}$. Using the completeness of $L_2(X_1)$, we can conclude that $(g_n)$ converges to some $g\in L_2(X_1)$. Hence $h=\widehat{g}\in G$, where $\hat{g}(t)=g(t)$ for $t\in X_1$ and $\hat{g}(t)=0$ otherwise.
		
Now
		\begin{align*}
		m(M_f|_{G})=&\inf\{\||M_f|_{G}(\widehat{g})\|:\,\widehat{g}\in S_{G}\}\\
		=&\inf\{\|M_{f_1}g\|:\,g\in L_2(X_1)\}\\
		=&m(M_{f_1}).
		\end{align*}
		As $M_f\in\mathcal{AM}(L_2(X))$, there exist $\hat{g_0}\in S_{G}$ such that $\|M_f|_{G}\hat{g_0}\|=m(M_f|_{G})$. Thus we get $\|M_{f_1}\hat{g_0}|_{X_1}\|=m(M_{f_1})$. Hence $M_{f_1}$ is minimum attaining. Again by (2), we get $ A_1\in\Sigma$ with $\mu(A_1)>0$ and $|f(t)|=\text{ess inf}(f_1)$ for all $t\in A_1$, also $A_0\cap A_1=\emptyset$.
		
		Continuing this way, we get a sequence of sets $(A_i)\subseteq\Sigma$ with $m(A_i)>0$, $A_i\cap A_j=\emptyset$ if $i\neq j$ and $|f(t)|=\text{ess inf}(f_i)$ for all $t\in A_i$, where $f_i=f$ with domain $X\setminus\left({\cup}_{j=0}^{i-1}A_j\right)$.
		
		Case(1): There exist $n_0\in\mathbb{N}$ such that $\mu\left(X\setminus\left({\cup}_{j=0}^{n_0}A_j\right)\right)=0$ then $X_0=X\setminus\left({\cup}_{j=0}^{n_0}A_j\right)$ and $M_{|f|}=\underset{i=0}{\overset{n_0}{\oplus}} M_{\text{ess inf}(f_i)}$.
		
		Case(2): We have the infinite sequence $(A_i)\subseteq\Sigma$. Since $(\text{ess inf}(f_i))$ is a monotonically increasing sequence and bounded above by $\|f\|_{\infty}$, so it is convergent. Let $(\text{ess inf}(f_n))$ converges to $\alpha$. Now set $X_0=X\setminus\left(\cup_{i=0}^{\infty} A_i\right).$ If $\mu(X_0)=0$ then we are done. Otherwise we repeat the above steps for $X_0$ and again get a sequence of sets $(B_j)$ with $\mu(B_j)>0$ and $|f(t)|=\text{ess inf}(g_j)$ for all $t\in B_j$, where $g_j=f$ with domain ${X_0\setminus\left(\underset{k=0}{\overset{j-1}{\cup }B_k}\right)}$. As in case (1), there exists a stage $m_0$ such that $\mu(X_0\setminus(\underset{k=0}{\overset{m_0}{\cup}}B_k))=0.$ Otherwise $\sigma_{ess}(M_{|f|})$ will have more than one element, which is a contradiction to the hypothesis that $M_f\in\mathcal{AM}(L_2(X))$. By taking $$X_0=X\setminus\left\{\left(\underset{i=0}{\overset{\infty}{\cup}}A_i\right)\cup\left(\underset{j=0}{\overset{m_0}{\cup}}B_j\right)\right\},$$ we get the result.

		Conversely, let $(A_i)\subseteq\Sigma$ be a sequence satisfying the given condition. First we will show that $M_f$ is minimum attaining. We have $\mu(A_0)>0$ and $|f(t)|=\text{ess inf}(f)$ for all $t\in A_0$. Choose a subset $B_0$ of $A_0$ such that $0<\mu(B_0)<\infty$. Consider $g=\frac{\chi_{B_0}}{\sqrt{\mu(B_0)}}$, we get $\|M_fg\|_2=\text{ess inf}(f)$.
		
		Now to show that $M_f\in\mathcal{AM}(L_2(X))$, consider a non-trivial closed subspace $E\subseteq L_2(X)$. Let $\mathcal{F}=\{B\in\Sigma:g=0\,a.e.\text{ on }B, \forall\,g\in E\}$ and define a relation $\mathtt{\sim}$ on $\mathcal{F}$ as,  $A\mathtt{\sim} B$ if $A\subseteq B$, for $A,B\in \mathcal{F}$. Now $(\mathcal{F},\mathtt{\sim})$ is a partially ordered set. Using Zorn's Lemma $\mathcal{F}$ has a maximal element, say $B_0$. Then $\mu(X\setminus B_0)>0$, otherwise $E$ will be the trivial space. Let $i_0$ be the smallest natural number for which $\mu((X\setminus B_0)\cap A_{i_0})>0$. So there exist $g_{i_0}\in E$ such that $g_{i_0}\neq 0$ a.e. on $(X\setminus B_0)\cap A_{i_0}$. Existence of such a $g_{i_0}$ is guaranteed, as if for every $g\in E,\,g=0$ a.e. on $C_0:=(X\setminus B_0)\cap A_{i_0}$, then $B_0\cup C_0\in \mathcal{F}$ which contradict the maximality of $B_0$ in $\mathcal{F}$. Now choose a measurable subset $D_0$ of $C_0$ such that $0<\mu(D_0)<\infty$ and consider
		\begin{align*}
		h(t)=&\frac{{\chi}_{D_0}(t)}{\sqrt{\mu(D_0)}}.\frac{g_{i_0}(t)}{|g_{i_0}(t)|}.
		\end{align*}
		Then
		\begin{align*}
		\|M_fh\|_2^2=\int_{D_0}|f(t)|^2d\mu=\text{ess inf}(f_{i_0})=m(M_f|_E)
		\end{align*}
		Thus $M_f|_E$ is minimum attaining. Hence $M_f\in\mathcal{AM}(L_2(X)).\qedhere$
	\end{enumerate}
\end{proof}

Using similar arguments as above, we can prove the following;

\begin{theorem}
Let $M_f$ be defined as in Example \ref{multplication_AMoperator}. Then we have the following;
\begin{enumerate}
\item $M_f\in \mathcal N(L^2(X))$ if and only if there exists a set $A\in \Sigma$ with $\mu(A)>0$ such that $|f(t)|=ess\,\sup(f)$, for all $t\in A$, where $ess\, \sup(f)=\inf\{\alpha\in\mathbb{R}:\mu(\{x\in X:|f(x)|>\alpha\})=0\}$.
\item $M_f\in \mathcal {AN}(L^2(X))$ if and only if there exist a sequence $(A_i)\subseteq \Sigma$ with $\mu(A_i)>0$ such that  $|f(t)|=ess\, \sup(f_i)$ for all $t\in A_i$ and $X=X_0\cup\left(\underset{i}{\cup}A_i\right)$, where $X_0\in\Sigma$ with $\mu(X_0)=0$ and $f_i=f$ with domain $X\setminus\left(\displaystyle \cup_{j=1}^{i-1}A_j\right)$.
\end{enumerate}
\end{theorem}
We close this section with the following question.
\begin{question}
	When a paranormal $\mathcal{AM}$-operator is normal?
\end{question}

\end{document}